\documentclass[a4paper,10pt]{amsart}
\usepackage{amstext, amsfonts,amsthm, amsmath, amssymb}
\usepackage{mathrsfs}
\usepackage[utf8]{inputenc}
\usepackage[colorlinks,linkcolor=blue,hyperindex]{hyperref}
\usepackage[justification=centering]{caption}
\usepackage{gensymb}
\usepackage{xfrac}  
\usepackage{multicol}
\usepackage{xcolor}
\usepackage{graphicx}
\usepackage{subfigure}
\usepackage{indentfirst}
\usepackage{caption} 
\usepackage{enumerate}
\usepackage[english]{babel}
\usepackage{fancyhdr}
\usepackage{refcount}
\usepackage{fullpage}
\usepackage{mathtools}
\usepackage{enumitem}
\usepackage{tabularx}
\usepackage{booktabs}
\usepackage{placeins}

\renewenvironment{abstract}
{\quotation\small\noindent\rule{\linewidth}{.5pt}\par\smallskip
	{\bfseries\abstractname.}}
{\par\noindent\rule{\linewidth}{.5pt}\endquotation}

\newtheorem{The}{Theorem}[section]
\newtheorem{Cor}[The]{Corollary}
\newtheorem{Def}[The]{Definition}
\newtheorem{Lem}[The]{Lemma}
\newtheorem{Rem}[The]{Remark}
\newtheorem{Exam}[The]{Example}
\newtheorem{Prop}[The]{Proposition}

\newcommand{\re}{\ensuremath{\mathbb{R}}}
\newcommand{\co}{\ensuremath{\mathbb{C}}}

\newcommand{\fd}{\ensuremath{\mathbb{F}}}
\newcommand{\nm}{\ensuremath{\Vert}}
\newcommand{\nmr}{\ensuremath{\right\Vert}}
\newcommand{\nml}{\ensuremath{\left\Vert}}
\newcommand{\abf}{\ensuremath{\textbf{a}}}
\newcommand{\bbf}{\ensuremath{\textbf{b}}}
\newcommand{\cbf}{\ensuremath{\textbf{c}}}
\newcommand{\dbf}{\ensuremath{\textbf{d}}}

\newcommand{\fab}{\ensuremath{f_{\textbf{a},\textbf{b}}}}
\newcommand{\fabt}{\ensuremath{f_{\textbf{a},\tilde{\textbf{b}}}}}
\newcommand{\fcd}{\ensuremath{f_{\textbf{c},\textbf{d}}}}
\newcommand{\fdct}{\ensuremath{f_{\textbf{c},\tilde{\textbf{d}}}}}
\newcommand{\muab}{\ensuremath{\mu_{\abf,\bbf}}}
\newcommand{\mucd}{\ensuremath{\mu_{\cbf,\dbf}}}

\newcommand{\ibf}{\ensuremath{\textbf{i}}}
\newcommand{\ga}{\ensuremath{g_{\textbf{a}}}}
\newcommand{\gc}{\ensuremath{g_{\textbf{c}}}}
\newcommand{\Vab}{\ensuremath{X_{\abf,\bbf}}}
\newcommand{\Vcd}{\ensuremath{X_{\cbf,\dbf}}}
\newcommand{\Va}{\ensuremath{X_{\abf}}}

\DeclareMathOperator{\fl}{\ensuremath{fil}}
\DeclareMathOperator{\lcm}{\ensuremath{lcm}}

\renewcommand{\Re}{\operatorname{Re}}
\renewcommand{\Im}{\operatorname{Im}}

\NewDocumentCommand{\mybar}{ O{0.7} O{2pt} m }{
    \mathrlap{\hspace{#2}\overline{\scalebox{#1}[1]{\phantom{\ensuremath{#3}}}}}\ensuremath{#3}
}
\newcommand*\cj[1]{\mybar{#1}}

\makeatletter
\def\blfootnote{\gdef\@thefnmark{}\@footnotetext}
\makeatother

\title[Short title]{Lipschitz Geometry of Mixed Pham-Brieskorn Singularities}
\author{Inácio Rabelo}
\thanks{The author was supported by the Coordenação de Aperfeiçoamento de Pessoal de Nível Superior - Brasil (CAPES) - Finance Code 001.}
\address{Instituto de Ciências Matematicas e de Computação, Av. Trabalhador São-Carlense 400, Centro. Caixa Postal: 668 CEP 13560-970, São Carlos SP, Brasil}
\email{rabeloinacio@usp.br}

\numberwithin{equation}{section}

\begin{document}

\maketitle

\begin{center}
\dedicatory{\it \small Dedicated to the memory of Professor Maria Ruas (Cidinha)}
\end{center}

\begin{abstract}
    We give conditions for topological and bi-Lipschitz equivalences within a class of mixed singularities of Pham-Brieskorn type. As a consequence, we construct infinite families that are topologically trivial but have distinct bi-Lipschitz types. We also investigate this problem in the context of mixed surfaces defined by these singularities in the case of two complex variables, deriving conditions for inner, outer, and ambient bi-Lipschitz equivalences. In particular, we obtain an invariant of the subanalytic outer geometry of the associated mixed surfaces, which is determined by the exponents. 


\end{abstract}

\section{Introduction}

Let $\abf = (a_{1}, \dots, a_{n})$ and $\textbf{b} = (b_{1}, \dots, b_{n})$ be vectors of integers such that $0<a_{1} \le \dots \le a_{n}$, $b_{i} \ge 0$ for all $i$, and consider the function germs defined by
\begin{equation*}\label{eqfam}
    \begin{split}
        \fab &= z_{1}^{a_{1}+b_{1}}\cj{z}_{1}^{b_{1}} + \dots + z_{n}^{a_{n}+b_{n}}\cj{z}_{n}^{b_{n}}, \\
        \ga &= z_{1}^{a_{1}} + \dots + z_{n}^{a_{n}}.
    \end{split}
\end{equation*}
The polynomial $\ga$ is a Pham-Brieskorn polynomial and we shall call $\fab$ a \textit{mixed Pham-Brieskorn} polynomial. Seade and Seade-Ruas-Verjovsky first studied this type of mixed function in \cite{sead} and \cite{ruas3} as examples of real isolated singularities admitting a Milnor fibration on a sufficiently small sphere $\mathbb{S}_{\epsilon}^{2n-1}$ in the usual form $\theta_{\abf,\bbf} := \fab/\nm \fab \nm: \mathbb{S}_{\epsilon}^{2n-1}\setminus K_{\abf,\bbf} \longrightarrow \mathbb{S}^{1}$, where $K_{\abf,\bbf} = \mathbb{S}_{\epsilon}^{2n-1} \cap \fab^{-1}(0)$ is the link. The germs $\fab$ and $\ga$ are related as follows. There exists a germ of homeomorphism $\phi : (\co^{n},0) \longrightarrow (\co^{n},0)$ such that $\fab \circ \phi = \ga$. Moreover, let $\theta_{\abf} := \ga/\nm \ga \nm: \mathbb{S}_{\epsilon}^{2n-1}\setminus K_{\abf} \longrightarrow \mathbb{S}^{1}$ be the Milnor fibration of $\ga$, where $K_{\abf} = \mathbb{S}_{\epsilon}^{2n-1} \cap \ga^{-1}(0)$ is the link. In \cite{oka}, Oka proves a stronger assertion relating the fibrations $\theta_{\abf,\bbf}$ and $\theta_{\abf}$: there exists a fiber preserving diffeomorphism $\varphi: (\mathbb{S}_{\epsilon}^{2n-1}, K_{\abf,\bbf}) \longrightarrow (\mathbb{S}_{\epsilon}^{2n-1}, K_{\abf})$.

The existence of a fibration as above is known as \textit{strong Milnor condition}, which remounts to a long-standing question raised by Milnor in his classical book \cite{miln} concerning the generalization, for real analytic maps, of his classical fibration theorem for holomorphic functions. The existence problem is treated in \cite{church} and \cite{raim}, and Milnor fibration type theorems are the subject of \cite{rai1}, \cite{rai2}, \cite{rai3}, \cite{seacisn}, and the references therein, to cite a few. As a natural generalization of holomorphic functions, mixed functions inherit several properties from this context and have a prominent and central role in the area. We refer the reader to \cite{oka2} for a general account. Our primary objective here is to investigate mixed singularities and varieties from the Lipschitz geometry point of view. As will be exemplified later by the characteristics of the family above, the properties invariant by Lipschitz equivalence lie between the topological and analytic realms, which justifies this approach.

In this direction, we derive conditions on the vectors of exponents in the case of a pair $\fab$ and $\fcd$ being topologically and bi-Lipschitz equivalent, where the mixed polynomials and functions are taken as real analytic map germs $(\re^{2n},0) \longrightarrow (\re^{2},0)$. Within this family, our results show that the bi-Lipschitz type of $\fab$ is determined by $\abf$ and $\bbf$. The proof is based on a theorem due to Yoshinaga-Suzuki about the topological types of Pham-Brieskorn hypersurfaces (Theorem \ref{etsu}) and a weaker version of the Splitting Lemma for bi-Lipschitz equivalence (Theorem \ref{splitt2}). We also cover the case $a_{i} = 0$ for some, but not all, indices $i$. A major feature of this case is that the singularity is not isolated. Nonetheless, the existence of Milnor fibrations and other properties was investigated in \cite{raf}. 

The inner bi-Lipschitz geometry of semialgebraic surfaces in the local case was classified by L. Birbrair in \cite{levsu}. Using the contact of semialgebraic arcs, we can determine the associated invariants and then the inner geometry of the mixed surfaces defined by $\fab : (\co^{2},0) \longrightarrow (\co,0)$. In most cases, they are inner equivalent to the standard cone in $\re^{3}$, also called \textit{metric cone}. Furthermore, we derive necessary conditions for the equivalence of the pairs $\fab^{-1}(0),\ga^{-1}(0)$ and $\fab^{-1}(0),\fcd^{-1}(0)$ relative to the three possible bi-Lipschitz equivalences of sets. Inspired by the case of complex plane curves, we shall see that the exponents $\abf$ and $\bbf$ determine an invariant, which we denote by $\muab$, of the outer geometry of the associated mixed surfaces. We also investigate the Lipschitz normal embedding property for these surfaces. This aims to understand the bi-Lipschitz properties of mixed functions and varieties through the relations with well-established complex cases and their simpler representatives. In the complex case and any dimension, this question is studied in \cite{samje}, where the authors proved that the multiplicity of Pham-Brieskorn varieties is preserved under bi-Lipschitz equivalence.

The article is organized as follows. Section \ref{prelip} recalls basic definitions of Lipschitz Geometry of analytic sets and singularities. We state results on invariants and finite determinacy for weighted homogeneous map germs. In Section \ref{topeq}, we study conditions on the exponents to have topological equivalence and, in particular, topological submersions. In Sections \ref{lipeq} and \ref{inff}, we apply these results to study the bi-Lipschitz equivalence problem and construct infinite families with the same topological, but distinct Lipschitz types. Finally, Section \ref{mixsuf} is devoted to the study of the mixed surfaces defined by $\fab: (\co^{2},0) \longrightarrow (\co,0)$.

We address a last remark. One may consider the polynomials $\fab$ with coefficients, that is, 
\begin{align*}
    f_{\abf,\bbf,\lambda} = \lambda_{1}z_{1}^{a_{1}+b_{1}}\cj{z}_{1}^{b_{1}} + \dots + \lambda_{n}z_{n}^{a_{n}+b_{n}}\cj{z}_{n}^{b_{n}},
\end{align*}
where $\lambda = (\lambda_{1}, \dots, \lambda_{n}) \in \co^{n}$ with $\lambda_{i} \neq 0$ for all $i$. If $a_{i} \ge 1$ for all $i$, then \cite[Lemma 8]{oka} states that up to a linear coordinate change, we may suppose the coefficients equal to $1$. If $a_{i} = 0$ for some indices $i$, this is no longer true, but it is immediate to verify that all conditions on the vectors of exponents derived in our theorems still hold. In this case, by using real coefficients, it is possible to conclude that the vector of coefficients does not determine the bi-Lipschitz type of $f_{\abf,\bbf, \lambda}$. Namely, there are bi-Lipschitz equivalent pairs $f_{\abf,\bbf,\lambda}, f_{\abf,\bbf, \gamma}$ for which $\lambda \neq \gamma$, where $\lambda, \gamma \in \re^{n}$ and $\lambda_{i},\gamma_{i} \neq 0$.

\section{Lipschitz geometry}\label{prelip}

\subsection{Singularities and analytic sets}

We recall the basic definitions of Lipschitz Geometry of sets and singularities. We refer the reader to the introductory text \cite{ruas2}. Let $(X,0) \subset \re^{n}$ be the germ of an analytic set. There exist two natural geometric structures on $(X,0)$ defined from the Euclidean metric. The \textit{outer metric} is the induced metric, that is, $d_{o}(x,y) = \nm x-y \nm$ for all $x,y \in X$. The \textit{inner metric} is the length metric defined by $d_{i}(x,y) = \inf \ell(\gamma)$ among all the rectifiable arcs $\gamma$ in $X$ connecting $x, y \in X$, where $\ell(\gamma)$ denotes its length. Given two germs of analytic sets $(X,0) \subset \re^{n}$, $(Y,0) \subset \re^{m}$, we can consider the following equivalence relations.
\begin{itemize}
    \item The inner equivalence: if there exists a homeomorphism $\phi:(X,0) \longrightarrow (Y,0)$ which is bi-Lipschitz with respect to the inner metric.
    \item The outer equivalence: if there exists a homeomorphism $\phi:(X,0) \longrightarrow (Y,0)$ which is bi-Lipschitz with respect to the outer metric.
    \item Ambient equivalence: if $m = n$ and there exists a bi-Lipschitz homeomorphism $\phi:(\re^{n},0) \longrightarrow (\re^{n},0)$ such that $\phi(X) = Y$.
\end{itemize}

Notice that the ambient equivalence implies the outer. By \cite[Remark 1.2]{tibar}, the outer equivalence also implies the inner. Now, we consider the bi-Lipschitz equivalence of smooth map germs or singularities. In what follows, $\fd = \re, \co$

Let $C(n,p)$ denote the set of smooth map germs $f: (\fd^{n},0) \longrightarrow (\fd^{p},0)$ and $J^{k}(n,p)$ the set of $k$-jets of elements of $C(n,p)$. Two germs $f,g \in C(n,p)$ are bi-Lipschitz-equivalent if there exists a germ of Lipschitz homeomorphism $\phi$ such that its inverse $\phi^{-1}$ is also a Lipschitz homeomorphism and $g = f \circ \phi^{-1}$. A germ $f \in C(n,p)$ is $k$-bi-Lipschitz determined if there exist an integer $k$ such that $f$ is bi-Lipschitz-equivalent to all germs $g$ such that $j^{k}_{0}(g) = j^{k}_{0}(f)$, where $j^{k}_{0}(g)$ and $j^{k}_{0}(f)$ are the $k$-jets at the the origin of $g$ and $f$, respectively. If $f$ is bi-Lipschitz determined for some $k$, we say that $f$ is bi-Lipschitz finitely determined. We refer to topological equivalence if $\phi$ is only a homomorphism. These relations are also known as \textit{right equivalence} in the literature. In the rest of the article, by equivalence (respectively, topological equivalence) and finite determinacy, we mean the previous equivalences under bi-Lipschitz homeomorphisms (respectively, homeomorphisms).

The Taylor expansion at the origin of an analytic function germ $f : (\fd^{n},0) \longrightarrow (\fd,0)$ can be written as $f = \sum_{i=k}^{\infty}f_{k}$, where $f_{i}$ is a homogeneous polynomial of degree $k$ for each $i$. The lowest degree $k$ is called the \textit{multiplicity} of $f$, denoted by $m_{f}$, and the associated homogeneous polynomial is denoted by $H_{f}$. We have the following results.

\begin{The}[Section 4, \cite{sara} - Corollary, \cite{trot}] \label{inva}\hfill
\begin{enumerate}
    \item If $f,g : (\fd^{n},0) \longrightarrow (\fd^{p},0)$ are bi-Lipschitz equivalent smooth maps germs, then the rank of $f$ at $0$ is equal to the rank of $g$ at $0$.
    \item If $f,g : (\fd^{n},0) \longrightarrow (\fd,0)$ are bi-Lipschitz equivalent smooth function germs, then $m_{f} = m_{g}$. Moreover, $H_{f}$ and $H_{g}$ are also bi-Lipschitz equivalent as well as their singular sets $\Sigma(H_{f})$ and $\Sigma(H_{g})$.
\end{enumerate}
\end{The}

An alternative proof for the first assertion of the second item can be seen in \cite[Lemma 5.12]{ruas2}. We recall the definition of the tangent cone of an analytic set to introduce the invariance theorem. This notion generalizes the tangent space at a smooth point of a real or complex variety.

    Let $X \subset \re^{n}$ be an analytic set and $x_{0} \in \overline{X}$. A vector $v \in \re^{n}$ is a \textit{tangent vector} of $X$ at $x_{0}$ if there exists a sequence $(x_{i}) \subset X\setminus \{x_{0}\}$ and a sequence of positive real numbers $(t_{i})$ such that $x_{i} \longmapsto x_{0}$ and $v = \lim_{i \to \infty} \frac{1}{t_{i}} (x_{i} - x_{0})$.
\begin{Def}[Definition 2.1, \cite{sam}]
    The tangent cone of $X$ at $x_{0}$ is the set
    $$C(X,x_{0}) = \left\{ v \in \re^{n} : v \;\,\text{is a tangent vector of} \;\, X\, \text{at} \;\, x_{0} \in \overline{X} \right\}.$$
\end{Def}
We usually refer to $C(X,x_{0})$ as the \textit{geometric tangent cone} of $X$ at $x_{0}$. If $X$ is a smooth manifold in a neighborhood of $x_{0}$, it is clear that the tangent cone coincides with the tangent space. There is also an alternative description in terms of velocity of arcs in the set as follows (see \cite[Remark 2.2]{sam}): the tangent cone is the set of vectors $v \in \re^{n}$ for which there exists a continuous semialgebraic curve $\alpha: [0,\epsilon) \longrightarrow \re^{n}$ such that $\alpha(0) = x_{0}$, $\alpha(0,\epsilon) \subset X$, and $\alpha(t) - x_{0} = tv + o(t)$, where $o(t)$ denotes higher degree terms. A second characterization of tangent vectors is the following.

\begin{Lem}[\cite{vico}, Lemma 2.3]\label{chale}
    Let $X$ be a real analytic variety and $v \in \re^{n}$ a non-zero vector. Then $v \in C(X,0)$ if and only if there exists a sequence $(t_{k}) \subset \re$ of positive integers such that $\lim_{k \to \infty}t_{k} = 0$ and
    \begin{align*}
        \lim_{k \to \infty}\frac{d(t_{k}v,X)}{t_{k}} = 0,
    \end{align*}
    where $d(t_{k}v,X)$ is the usual Euclidean distance from the point $t_{k}v \in \re^{n}$ to the analytic set $X$.
\end{Lem}

Let $\mathcal{O}_{n}$ be the ring of analytic function germs at the origin in $\re^{n}$. Let $I(X) \subset \mathcal{O}_{n}$ be the ideal generated by the germs of analytic functions vanishing at $X$. Given $f \in I(X)$, let $H_{f}$ denote the lowest degree homogeneous polynomial of $f$ and $I_{0}$ the ideal generated by $H_{f}$ for all $f \in I(X)$. The \textit{algebraic tangent cone} of $X$ at $0$ is
$$C_{a}(X,0) = \{ x \in \re^{n} : H_{f}(x) = 0 \; \forall \; f \in I(X)\}.$$
A classical result due to Whitney states that the algebraic tangent cone and the tangent cone coincide in the complex case. In general, we have the following properties:
\begin{Prop}[Remark 2.2, \cite{vico}]\label{rem0}
    Let $X, Y \subset \re^{n}$ be analytic sets. The following inclusions hold:
    \begin{enumerate}
        \item $C(X\cap Y,0) \subset C(X,0) \cap C(Y,0)$.
        \item $C_{a}(X\cap Y,0) \subset C_{a}(X,0) \cap C_{a}(Y,0)$.
        \item $C(X,0) \subset C_{a}(X,0)$.
    \end{enumerate}
\end{Prop}

\begin{Rem}\label{rem1}
        \normalfont
        Let $X = V(I) \subset \re^{n}$ be an analytic set defined by an ideal $I \subset \mathcal{O}_{n}$. If $I(X) = \{ g \in \mathcal{O}_{n}: g(x) = 0 \; \forall \; x \in X\}$ and $I_{0}(X)$ and $I_{0}$ are the ideals generated by all the forms of elements in $I(X)$ and $I$, respectively, it holds that $I_{0} \subset I_{0}(X)$ and thus $V(I_{0}(X)) = C_{a}(X,0) \subset V(I_{0})$.
    \end{Rem}

    The main result in \cite{sam} is the invariance of the tangent cone under bi-Lipschitz homeomorphisms.
        \begin{The}[Theorem 3.2, \cite{sam}]\label{tsam}
            Let $X, Y \subset \re^{n}$ be subanalytic sets. If the germs $(X,x_{0})$ and $(Y,y_{0})$ are outer bi-Lipschitz equivalent, then the tangent cones $(C(X,x_{0}),x_{0})$ and $(C(Y,y_{0}), y_{0})$ are also outer bi-Lipschitz equivalent.
        \end{The}

\subsection{Semialgebraic surfaces}

   We recall in this section the inner bi-Lipschitz classification of semialgebraic surfaces $X \subset \re^{n}$ with an isolated singularity. We refer the reader to \cite{levsu} and \cite{levwei} for details. They are associated with the so-called $\beta$-horns as follows.
   
     \begin{Def}
            The $\beta$-horn is the semialgebraic set defined by
            $$H_{\beta} = \left\{ (x,y,z) \in \re^{3} : \left(x^{2}+y^{2}\right) = z^{2\beta}, z\ge 0\right\},$$
            where $\beta \ge 1$ is a rational number.
        \end{Def}
    If $\beta = 1$ the horn $H_{1}$ is called the \textit{metric cone}. Moreover, $H_{\beta}$ is inner equivalent to $H_{\beta'}$ if and only if $\beta = \beta'$. In \cite{levsu}, L. Birbrair gives the following classification.

    \begin{The}\label{inde1}
        Let $(X,0) \subset \re^{n}$ be a semialgebraic surface germ with an isolated singularity at the origin and connected link. Then $(X,0)$ is inner bi-Lipschitz equivalent to a $\beta$-horn $H_{\beta}$ for some $\beta \ge 1$.
    \end{The}
    
    We present now a useful result to compute the index $\beta$. First, recall that a \textit{semialgebraic arc} at the origin in $\re^{n}$ is the image of a non-constant semialgebraic map $\gamma : [0,\epsilon) \longrightarrow \re^{n}$ such that $\gamma(0) = 0$. Let $\gamma_{1},\gamma_{2}$ be two semialgebraic arcs. A semialgebraic function $f(t)$ has an expansion $f(t) = at^{p} + o(p)$, where $a \neq 0$ and $ p$ is rational. The order of contact $\lambda(\gamma_{1},\gamma_{2})$ is defined as the order of the expansion of the semialgebraic function $\rho(t) = \nm \gamma_{1}(t) - \gamma_{2}(t) \nm$. For details, we refer the reader to \cite[Lemma 2]{tibar}. In particular, let $\gamma_{i} = u_{i}t + o(t)$, with $u_{1}, u_{2}$ linearly independent vectors and $i=1,2$, then the order of contact is $1$.
    
    \begin{The}[Theorem 4.1, \cite{levwei}]\label{inde}
        Let $(X,0) \subset \re^{n}$ be a semialgebraic surface germ which is inner bi-Lipschitz equivalent to a germ of a $\beta$-horn. Then
        \begin{align*}
            \beta = \inf\left\{ \lambda(\gamma_{1},\gamma_{2}) : \gamma_{1},\gamma_{2} \; \text{are semialgebraic arcs in $X$ such that} \; \gamma_{1}(0) = \gamma_{2}(0) = 0\right\}.
        \end{align*}
    \end{The}

\subsection{Determinacy of weighted homogeneous germs}

In this section, we introduce finite determinacy results for the weighted homogeneous case.

\begin{Def}
    \hfill
    \begin{enumerate}
        \item Given $(r_{1}, \dots, r_{n}; d_{1}, \dots, d_{p})$, where $r_{i}, d_{i} \in \mathbb{Q}^{+}$, a map germ $f : (\re^{n},0) \longrightarrow (\re^{p},0)$ is weighted homogeneous of type $(r_{1}, \dots, r_{n}; d_{1}, \dots, d_{p})$ if for all $\lambda \in \re \setminus \{0\}$:
\begin{align*}
    f(\lambda^{r_{1}}x_{1}, \dots, \lambda^{r_{n}}x_{n}) = (\lambda^{d_{1}}f_{1}(x), \dots, \lambda^{d_{p}}f_{p}(x)).
\end{align*}
\item For each monomial $x^{\alpha} = x_{1}^{\alpha_{1}}\dots x_{n}^{\alpha_{n}}$ in $C(n,1)$ we define $\fl(x^{\alpha}) = \sum_{i=1}^{n}r_{i}\alpha_{i}$.
\item We define a filtration in the ring $C(n,1)$ via the following function:
\begin{align*}
    \fl(f) = \inf_{\alpha}\left\{ \fl(x^{\alpha}) : \left( \frac{\partial^{\alpha}f}{\partial x^{\alpha}}\right)(0) \neq 0 \right\}.
\end{align*}
\item For any map germ $f = (f_{1}, \dots, f_{p}) \in C(n,p)$ we call $\fl(f) = (d_{1}, \dots, d_{p})$, where $d_{i} = \fl(f_{i})$ for each $i=1, \dots, p$. 
    \end{enumerate}
\end{Def}                    

Based on a particular version in \cite{ruas5}, the next result is stated in \cite{saia}.

\begin{The}\label{def}
    Let $f : (\re^{n},0) \longrightarrow (\re^{p},0)$ be a weighted homogeneous map germ with an isolated singularity at the origin. Suppose that $f$ is of type $(r_{1}, \dots, r_{n} ; d_{1}, \dots, d_{p})$, where $r_{1} \le r_{2} \le \dots \le r_{n}$. Let $f_{t}(x) = f(x) + t\Theta(x,t)$ be a smooth deformation of $f$ with $\Theta = (\Theta_{1}, \dots, \Theta_{p}) : (\re^{n},0) \longrightarrow (\re^{p},0)$. The following statements hold:
    \begin{enumerate}
        \item If $\fl(\Theta_{i}) > d_{i} + r_{n} - r_{1}$, then $f_{t}$ admits a bi-Lipschitz trivialization along $I = [0,1]$.
        \item If $\fl(\Theta_{i}) = d_{i} + r_{n} - r_{1}$, then $f_{t}$ is bi-Lipschitz trivial for small $t$.
    \end{enumerate}
\end{The}
    The next example shows that the inequalities of the hypothesis above are sharp.
    \begin{Exam}
        \normalfont
        Consider the function germ $f(x,y) = x^{2} + y^{2}$ and its deformation
        \begin{align*}
            f_{t}(x,y) = f(x,y) + t(-x^{2} - y^{2} + x^{2k+1} + y^{2k+1}),
        \end{align*}
        where $t \in [0,1]$. Notice that $f_{1}^{-1}(0)$and $f^{-1}(0)$ have distinct topologies, hence the deformation is not topologically trivial along $[0,1]$.
     \end{Exam}

    \section{Topological equivalence}\label{topeq}

    We detail in this section the topological relation between the mixed polynomials $\fab$ and Pham-Brieskorn polynomials $\ga$ and some consequences. This allows us to derive conditions on the vector of exponents in the case of topological equivalence. In particular, whether $\fab$ is a topological submersion.

    \begin{Prop}[Theorem 4.1, \cite{ruas3}]\label{topfor}
        Let $b_{i} \ge 0$ for all $i = 1, \dots, n$. The following statements hold:
        \begin{enumerate}
            \item If $a_{i} \ge 1$ for all $i$, then $\fab$ is topologically equivalent to the holomorphic Pham-Brieskorn polynomial $\ga$.
            \item If $a_{i} = 0$ for all $i = 1, \dots, k$ and $a_{i} \ge 1$ for all $i = k+1, \dots, n$, then $\fab$ is topologically equivalent to $\fabt = \sum_{i=1}^{k}\nm z_{i}\nm^{2b_{i}} + \sum_{i=k+1}^{n}z_{i}^{a_{i}}$, where $1 \le k < n$ and $\tilde{\textbf{b}} = (b_{1}, \dots, b_{k}, 0, \dots, 0)$.
        \end{enumerate}
        We call $\ga$ and $\fabt$ the topological normal forms of $\fab$, respectively.
    \end{Prop}
    \begin{proof}
    Let $D = \{ z_{1}\dots z_{n} = 0\}$ and define the map $\phi : \co^{n}\setminus D \longrightarrow \co^{n}\setminus D$ by $\phi(z) = w$, where $w_{i} = z_{i}\nm z_{i} \nm^{2b_{i}/a_{i}}$, for all $i = 1, \dots, k$, and $w_{i} = z_{i}$ if $a_{i} = 0$ for all $i = k+1, \dots, n$. Its inverse map is $\phi^{-1}(w) = z$, where
    $$ \begin{cases}
        z_{i} = w_{i} & \text{if} \;i = 1, \dots, k,\\
        z_{i} = \frac{w_{i}}{\nm w_{i} \nm^{\frac{2b_{i}}{a_{i}+2b_{i}}}} & \text{if}\; i = k+1, \dots, n, 
    \end{cases}
    $$
    Since $2b_{i} < a_{i} + 2b_{i}$ provided that $a_{i} \ge 1$, it holds that $\lim_{w_{i} \to 0}z_{i} = 0$. Whence $\phi$ extends to a homeomorphism of $\co^{n}$ by mapping $z_{i} = 0$ to $w_{i} = 0$ for each $i = 1, \dots, n$, and satisfies $\ga \circ \phi = \fab$, in the case $k = 0$, or $\fabt \circ \phi = \fab$ if $k \neq n$.
    \end{proof}
    However, note that on each coordinate subspace $\{z_{i} = 0\}$, $\phi$ is never a diffeomorphism. Recall that two germs of sets $X, Y \subset \re^{n}$ at the origin have the same \textit{topological type}, also called embedded topological type if there exists a germ of homeomorphism $\phi : (\re^{n},0) \longrightarrow (\re^{n},0)$ such that $\phi(X) = Y$.
    \begin{The}[Theorem, \cite{etsu}]\label{etsu}
        Let $g_{\textbf{a}},g_{\cbf} : (\co^{n},0) \longrightarrow (\co,0)$ be two Pham-Brieskorn polynomials, where $a_{i},c_{i} \ge 2$ for all $i = 1, \dots, n$. Then the hypersurface germs defined by $g_{\textbf{a}}$ and $g_{\cbf}$ have the same topological type if and only if $\abf = \cbf$.
    \end{The}
   
    This result will also describe the topological properties of $\fab$. For instance, if $a_{1} = 1$, then $\fab$ is a topological submersion. We shall prove that the converse holds, even in the case $a_{i} = 0$ for some, but not all, indices $i$. For this purpose, we must return to the holomorphic context.

    The well-known Zariski conjecture concerns the invariance of the multiplicity of holomorphic function germs under a homeomorphism of their zero sets. In \cite{le}, Lê Dung Trang proved that the Betti numbers of the associated Milnor fiber are preserved under homeomorphisms. On the other hand, A'Campo proved in \cite{camp} that regularity at a point implies that the homology of the fiber is trivial (see \cite{eyr} for details). Since regularity means multiplicity equals one, the conjecture holds in this case:
    \begin{Lem}[Theorem 11.3, \cite{eyr}]\label{l0}
        Let $f,g : (\co^{n},0) \longrightarrow (\co,0)$ be reduced germs of holomorphic functions. If $f$ and $g$ are topologically equivalent and the multiplicity $m_{f} = 1$, then $m_{g} = 1$. In other words, the multiplicity is equal to one is preserved under homeomorphisms.
    \end{Lem}
We obtain the following characterization of topological submersions.
\begin{Prop}\label{submfam}
Let $a_{i} \ge 0$ for all $i$ and $a_{j} \ge 1$ for some $j$. Then the mixed function $\fab$ is topologically equivalent to a submersion $(\re^{2n},0) \longrightarrow (\re^{2},0)$ if and only if $a_{i} = 1$ for some $i$.
\end{Prop}
    \begin{proof}
        Consider the topological normal forms of $\fab$. If $a_{i}=1$ for some $i$, then the Jacobian matrix of $\ga$ or $\fabt$ at $0$ contains a non-vanishing minor of order 2 associated with $z_{i}$ and thus it is a submersion. Conversely, in the case $a_{i} \ge 1$ for all $i$, the assertion is a consequence of Lemma \ref{l0}. If $a_{j} = 0$ for some subset of indices, we may suppose $\fabt = \sum_{i=1}^{k}\nm z_{i} \nm^{2b_{i}} + \sum_{i=k+1}^{n}z_{i}^{a_{i}}$, where $k < n$ and $a_{i} \ge 1$ for all $i = k+1, \dots, n$. In this case, since the function $\fabt$ restricted to the variables $z_{i}$ for $i \le k$ is real and singular at the origin for any coordinate system, the complex function $\sum_{i=k+1}^{n}z_{i}^{a_{i}}$ is a topological submersion. Therefore, the same lemma applies and we conclude that $a_{i} = 1$ for some $i=k+1, \dots, n$.
    \end{proof}  
    We can state the following.
    
    \begin{The}\label{cortop}
        Let $b_{i}, d_{i} \ge 0$ for all $i = 1, \dots, n$ and suppose that $\fab$ is topologically equivalent to $\fcd$. Then the following statements hold:
        \begin{enumerate}
            \item If $a_{i}, c_{i} \ge 2$, then $\abf = \cbf$.
            \item If $a_{i} = 0$ for all $i = 1, \dots, k$ and $a_{i} \ge 2$ for all $i = k+1, \dots, n$, then $c_{i} = 0$ for all $i=k+1, \dots, n$.
            \item If $a_{1} = 1$, then $c_{1} = 1$. 
        \end{enumerate}
    \end{The}
    \begin{proof}
        The first assertion results from applying Theorem \ref{etsu} on the topological normal forms $\ga$ and $\gc$. In the second case, suppose that there exists a germ of homeomorphism $\psi : (\co^{n},0) \longrightarrow (\co^{n},0)$ such that $\fabt \circ \psi = \fdct$. We may suppose
        \begin{align*}
            \fabt &= \sum_{i=1}^{k}\nm z_{i} \nm^{2b_{i}} + \sum_{i=k+1}^{n}z_{i}^{a_{i}}, \\
            \fdct &= \sum_{i=1}^{l}\nm z_{i} \nm^{2d_{i}} + \sum_{i=l+1}^{n}z_{i}^{c_{i}}.
        \end{align*}
       Since the imaginary parts of $\fabt$ and $\fdct$ are topologically equivalent, the codimensions of the zero sets determined by $\Im(\fabt)$ and $\Im(\fdct)$ are the same, and thus $l = k$. The last case is a consequence of Proposition \ref{submfam}, since the topological normal forms are submersions.
    \end{proof}
    In the first case of the above theorem, if $\abf = \cbf$, it is clear that the pair is topologically equivalent. Moreover, if $a_{i} = 1$ for some $i$, then they are equivalent if and only if $c_{i} = 1$.

\section{Bi-Lipschitz equivalence}\label{lipeq}

To study the bi-Lipschitz equivalence of pairs of Pham-Brieskorn polynomials, we prove a result on the multiplicity of germs in separable variables. This is a weaker version of the Splitting Lemma for bi-Lipschitz equivalence of singularities.

  \begin{The}\label{splitt2}
    Let $F,G : (\fd^{2n},0) \longrightarrow (\fd,0)$ be a pair of smooth function germs of the following forms: 
        \begin{align*}
            &\begin{cases}
               F(x_{1}, y_{1}, \dots, x_{n},y_{n}) = f_{1}(x_{1},y_{1}) + \dots + f_{n}(x_{n},y_{n}),\\
               G(x_{1}, y_{1}, \dots, x_{n},y_{n}) = g_{1}(x_{1}, y_{1}) + \dots + g_{n}(x_{n},y_{n}).
            \end{cases} 
       \end{align*}
    Suppose that:
    \begin{enumerate}
        \item For each $i = 1, \dots, n$, $f_{i}(x_{i},y_{i}),g_{i}(x_{i},y_{i}) : (\fd^{2},0)\longrightarrow (\fd,0)$ are homogeneous polynomials and have isolated singularity at the origin.
        \item $2 \le m_{f_{1}} < m_{f_{2}} < \dots < m_{f_{n}}$.
        \item $2 \le m_{g_{1}} < m_{g_{2}} < \dots < m_{g_{n}}$.
    \end{enumerate}
    If $F$ and $G$ are bi-Lipschitz equivalent, then $m_{f_{i}} = m_{g_{i}}$ for all $i = 1, \dots, n$.
\end{The}
    \begin{proof}
        The statement is a generalization of \cite[Theorem 5.1]{sara} and has a similar proof. Let $\phi = (\phi_{1}, \dots, \phi_{n})$ be a bi-Lipschitz homomorphism of $\fd^{2n}$ such that $F \circ \phi = G$, where $\phi_{i} : \fd^{2n} \longrightarrow \fd^{2}$ for each $i$. Let us denote its inverse by $\psi = (\psi_{1}, \dots, \psi_{n})$ and consider the sequences
        \begin{align*}
            \phi_{m} &= (\phi_{1,m}, \dots, \phi_{n,m}) = m.\phi\left(\frac{x_{1}}{m}, \dots, \frac{x_{n}}{m}\right), \\
             \psi_{m} &= (\psi_{1,m}, \dots, \psi_{n,m}) = m.\psi\left(\frac{x_{1}}{m}, \dots, \frac{x_{n}}{m}\right).
        \end{align*}
        By Arzelà–Ascoli theorem, there exists $\{m_{i}\} \subset \mathbb{N}$ such that $\phi_{m_{i}}, \psi_{m_{i}}$ converge to bi-Lipschitz homeomorphisms $d\phi = (d\phi_{1}, \dots, d\phi_{n})$ and $d\psi = (d\psi_{1}, \dots, d\psi_{n})$, respectively. Moreover, $d\phi$ is the inverse of $d\psi$. These limits can be interpreted as the derivatives of $\phi$ and $\psi$. By the conditions on $f_{i}, g_{i}$, the singular sets are
        \begin{align*}
            \Sigma(f_{i}) = \fd^{2} \times \dots \times \fd^{2} \times \underbrace{\{0\}}_{i} \times \fd^{2} \times \dots \times \fd^{2}.
        \end{align*}        
        By \cite[Lemma 4.6]{sara}, it holds that $m_{f_{1}} = m_{g_{1}}$ and $d\phi(\Sigma(f_{1})) = \Sigma(g_{1})$. This implies
        \begin{align}\label{restr}
            d\phi(p_{2}) = \left(0, d\tilde{\phi}_{2}(p_{2})\right),
        \end{align}
        where $p_{2} \in \{0\} \times \fd^{2(n-1)}$ and $d\tilde{\phi}_{2} = \left(d\phi_{2}, \dots, d\phi_{n}\right)$ is a bi-Lipschitz homeomorphism of $\{0\}\times \fd^{2(n-1)}$, which is the limit of the sequence above constructed from $\tilde{\phi}_{2} = (\phi_{2}, \dots, \phi_{n})$. We remark on the following identity (see \cite[Remark 4.4]{sara}):
        \begin{align}\label{eq}
            \begin{split}
            \lim_{m_{i}\to \infty}m_{i}^{m_{f_{2}}-1}DF_{\tilde{\phi}_{2}(0,y/m_{i})} &= \lim_{m_{i}\to \infty}\left( m_{i}^{m_{f_{2}}-1}(Df_{2})_{\tilde{\phi}_{2}(0,y/m_{i})} + m_{i}^{m_{f_{2}}-1}o(\nm y/m_{i} \nm)\right) \\
            &= \lim_{m_{i}\to \infty}\left((Df_{2})_{m_{i}\tilde{\phi}_{2}(0,y/m_{i})} + m_{i}^{m_{f_{2}}-1}o(\nm y/m_{i} \nm)\right) \\
            &= (Df_{2})_{d\tilde{\phi}_{2}(0,y)},     
            \end{split}
        \end{align}
        for all $(0,y) \in \fd^{2(n-1)}$. As a consequence of \cite[Lemma 4.5]{sara}, we have:        
        \begin{align}\label{ineq2}
            \nm (Df_{2})_{\tilde{\phi}_{2}(p_{2})}\nm &\le L\nm (DG)_{p_{2}}\nm, \\
            \nm (Dg_{2})_{\tilde{\psi}_{2}(p_{2})} \nm &\le L'\nm (DF)_{p_{2}} \nm,
        \end{align}
        where $L,L' > 0$ are constants and $p_{2} \in \{0\} \times \fd^{2(n-1)}$. We now proceed to prove that $m_{f_{2}} = m_{g_{2}}$ and $d\tilde{\phi}_{2}(\Sigma(f_{2})) = \Sigma(g_{2})$. Multiplying \eqref{ineq2} by $m_{i}^{m_{g_{2}-1}}$, we obtain
        \begin{align}\label{ineq3}
            \nml m_{i}^{m_{g_{2}}-1}\left(Df_{2}\right)_{\tilde{\phi}_{2}(0,y/m_{i})}\nmr \le L \nml m_{i}^{m_{g_{2}}-1}\left(DG\right)_{(0,y/m_{i})} \nmr.
        \end{align}

        Write $(DG)_{(0,y/m_{i})} = (Dg_{2})_{(0,y/m_{i})} + o(\nm y/m_{i}\nm)$, where $o(\nm y/m_{i}\nm)$ denotes higher order terms. Multiplying this equality by $m_{i}^{m_{g_{2}}-1}$ and taking limits, the right-hand side tends to $(Dg_{2})_{(0,y)}$ because $Dg_{2}$ is homogeneous of degree $m_{g_{2}}-1$. This implies \eqref{ineq3} is bounded provided $y$ is. We conclude the same for $m_{i}^{m_{f_{2}}-1}DF_{(0,y/m_{i})}$, which tends to $(Df_{2})_{(0,y)}$. By isolated singularity condition and the fact that $d\tilde{\phi}_{2}$ restricts as a homeomorphism of $\fd^{2(n-1)}$, there exists $y_{0} \neq 0$ such that $(Df_{2})_{(0,d\tilde{\phi}_{2}(0,y_{0}))} \neq 0$. If $m_{g_{2}} > m_{f_{2}}$, by \eqref{eq}, the following limit diverges and contradicts the boundness of \eqref{ineq3}:
        \begin{align*}
            \lim_{m_{i}\to\infty}\nml m_{i}^{m_{g_{2}-1}}(Df_{2})_{\tilde{\phi}_{2}(0,y_{0}/m_{i})}\nmr &= \lim_{m_{i}\to\infty}\nml m_{i}^{m_{g_{2}}-m_{f_{2}}}m_{i}^{m_{f_{2}}-1}(Df_{2})_{\tilde{\phi}_{2}(0,y_{0}/m_{i})}\nmr \\
            &= \lim_{m_{i}\to \infty}m_{i}^{m_{g_{2}}-m_{f_{2}}}\nml (Df_{2})_{d\tilde{\phi}_{2}(0,y_{0})}\nmr.
        \end{align*}
        Analougously for the case $m_{g_{2}} < m_{f_{2}}$ and we then conclude $m_{f_{2}} = m_{g_{2}}$. Identify now $\Sigma (f_{2}) \simeq \{0\} \times \Sigma (f_{2})$ and $\Sigma (g_{2}) \simeq \{0\} \times \Sigma (g_{2})$. We shall see that $d\tilde{\phi}_{2}(0,y)\left(\Sigma(f_{2})\right)$ $\subset \Sigma(g_{2})$ and $(d\tilde{\psi}_{2})(0,y)\subset \Sigma(f_{2})$. Since $m_{f_{2}}=m_{g_{2}}$, \eqref{ineq3} implies that
        \begin{align*}
            \nml (Df_{2})_{d\tilde{\phi}_{2}(0,y)} \nmr \le L \nml (Dg_{2})_{(0,y)}\nmr.
        \end{align*}
        If $p = (0,y) \in \Sigma(g_{2})$, then $(Df_{2})_{d\tilde{\phi}_{2}(0,y)} = 0$ and similarly for $g_{2}$ and $d\tilde{\psi}_{2}$. This proves both inclusions. These two facts imply that $d\tilde{\phi}_{2}$ restricts to $\fd^{2(n-2)}$ as a bi-Lipschitz homeomorphism as $d\phi$ in \eqref{restr}. For $i = 3, \dots, n$ and proceeding inductively, we consider the bi-Lipschitz restriction to $\fd^{2(n-i)}$ of the homeomorphism $d\tilde{\phi}_{2}$ and Equations \eqref{ineq2} for $f_{i}$ and $g_{i}$ at the points $\tilde{\phi}_{i}(p_{i})$ and $\tilde{\psi}_{i}(p_{i})$, respectively, to obtain that $m_{f_{i}} = m_{g_{i}}$ for all $i = 1, \dots, k$.    
    \end{proof}

    \begin{Rem}
       \normalfont
       Suppose that $m_{f_{i}} = \dots = m_{f_{i+k}}$ for some $k \ge 1$ and define $\tilde{f}_{i} = f_{i} + \dots + f_{i+k}$. Following the proof above, we have that $\nm (D\tilde{f}_{i})_{d\tilde{\phi}_{i}(0,y)}\nm \le L \nm (D\tilde{g}_{i})_{(0,y)}\nm$ for some $L > 0$, where $\tilde{g}_{i} = g_{i} + \dots + g_{i+l}$ and $m_{f_{i}} = m_{g_{i}} = \dots = m_{g_{i+l}}$ for some $l$. A reverse inequality also holds at a point $d\tilde{\psi}_{i}(0,y)$. Then the singular sets of $\tilde{f}_{i}$ and $\tilde{g}_{i}$ are related by a homeomorphism, and thus $l=k$. This shows that the argument may be applied to blocks of germs with distinct multiplicities, and the inequalities are not restrictive.
    \end{Rem}
    
Now, we can deduce that the bi-Lipschitz type of $\fab$ within the family of Pham-Brieskorn polynomials is determined by $\abf$ and $\bbf$, which can be seen as a generalization of Theorem \ref{etsu}. 

\begin{The}\label{class1}
Let $b_{i}, d_{i} \ge 0$ for all $i=1, \dots, n,$ and suppose that $\fab$ is bi-Lipschitz equivalent to $\fcd$. Then the following statements hold:
    \begin{enumerate}
        \item If $a_{i}, c_{i} \ge 2$ for all $i=1, \dots, n$, then $\abf = \cbf$ and $a_{i} + 2b_{i} = c_{\sigma(i)} + 2d_{\sigma(i)}$ for every $i = 1, \dots, k$ and some permutation $\sigma$.
        \item If $a_{i} = 0$ for all $i = 1, \dots, k$ and $a_{i} \ge 2$ for all $i = k+1, \dots, n$, then $c_{i} = 0$ for all $i=, \dots, k$, and $a_{i} +2b_{i} = c_{\sigma(i)} + 2d_{\sigma(i)}$ for all $i = 1, \dots, n$ and some permutation $\sigma$.
        \item If $a_{1} = 1$, then $c_{1} = 1$. Moreover, $a_{i} + 2b_{i} = c_{\sigma(i)} + 2d_{\sigma(i)}$ for every $i = 1, \dots, n$ and some permutation $\sigma$.
    \end{enumerate}
\end{The}
    \begin{proof}
        First, by Propositions \ref{submfam} and \ref{cortop}, the topological type determines $\abf$ and $\cbf$. Moreover, from the proof of Lemma \ref{l1}, the monomials $z_{i}^{a_{i}+b_{i}}\cj{z}_{i}^{b_{i}}$ have an isolated singularity at the origin, then the same holds for their real parts. We apply now Theorem \ref{splitt2} for the real parts of the germs, for which the multiplicities of the monomials coincide up to a permutation $\sigma$ applied in the indices of one of the functions, and the result follows.
    \end{proof}

For each $i$ in the first case, $a_{i} + 2b_{i} = c_{\sigma(i)} + 2d_{\sigma(i)}$ and $c_{\sigma(i)} = a_{j}$ for some $j$. Then $d_{\sigma(i)}$ is depends on $a_{i}, a_{j}$, and $b_{i}$, and thus $\dbf$ is determined by $\abf$ and $\bbf$. The next corollary shows that there is no bi-Lipschitz equivalence between the mixed and complex Pham-Brieskorn polynomials, except possibly for topological submersions.

\begin{Cor}
    Let $b_{1} \le \dots \le b_{n}$ and suppose that $\fab$ is bi-Lipschitz equivalent to $\gc$. Then the following statements hold:
    \begin{enumerate}
        \item If $a_{i} \ge 2$ for all $i$, then $\abf = \cbf$ and $\bbf = 0$, that is, $\fab = \gc$.
        \item If $a_{1} = 1$, then $c_{1} = 1$ and $a_{i} + 2b_{i} = c_{i}$ for all $i$. In particular, $b_{1} = 0$.
    \end{enumerate}
\end{Cor}
    \begin{proof}
        In the first item, considering the normal form $\ga$ for $\fab$, Theorem \ref{etsu} implies that $\abf = \cbf$ and Theorem \ref{class1} shows that $a_{i} + 2b_{i} = c_{i}$ for all $i$. Analogously for the second case, but now applying Proposition \ref{cortop}.    
    \end{proof}
    
    In particular, for the above cases, the mixed polynomial $\fab$ is bi-Lipschitz equivalent to its topological normal form $\ga$ if and only if $\bbf = 0$. 

\section{Families and deformations}\label{inff}

The criteria above allow us to construct families of mixed functions with the same topological type, but distinct Lipschitz types. This is generalized through the finite determinacy Theorem \ref{def} and is a direct consequence of Theorem \ref{class1}.

\begin{Prop}\label{corfam}
     Let $\abf$ be a fixed vector of integers, where $a_{i} \ge 0$ for all $i=1, \dots, n$ and $a_{j} \ge 1$ for some $j$. There exist subfamilies of $\Gamma_{\abf} = \{ f_{\abf,\bbf}: (\co^{n},0) \longrightarrow (\co,0): b_{1}, \dots, b_{n} \in \mathbb{N}\}$ that are topologically trivial and contain infinitely many distinct Lipschitz classes.
\end{Prop}
\begin{proof}
    If $a_{i} \ge 2$ for all $i$, then every $\fab$ is topologically equivalent to $\ga$, and the statement holds for the entire family. Suppose now that $a_{i} = 0$ for $i= 1, \dots k$ and $a_{i} \ge 2$ for $i = k+1, \dots, n$. Let $b_{1} \le \dots \le b_{k}$ be fixed. Then every $\fab$ is topologically equivalent to $\fabt$, where $\bbf = (b_{1}, \dots, b_{k}, b_{k+1}, \dots, b_{n})$. On the other hand, in both cases, $f_{\abf,\bbf}$ and $f_{\abf,\cbf}$ are bi-Lipschitz equivalent if and only if $\bbf = \cbf$, by Theorem \ref{class1}, provided that $b_{1} \le \dots \le b_{n}$ and $c_{1} \le \dots \le c_{n}$. Moreover, by the same result, if $a_{1} = 1$, then for $\bbf$ and $\cbf$ such that $b_{1} \neq c_{1}$, $f_{\abf,\bbf}$ and $f_{\abf,\cbf}$ are not bi-Lipschitz equivalent.
\end{proof}
This allows us to construct a variety of families with the same property.
    \begin{Lem}\label{l1}
        Let $\textbf{a}, \textbf{b}$ be integer vectors such that $a_{i},b_{i} \ge 1$ for all $i = 1, \dots, n$, $d = \lcm\{a_{i} + 2b_{i}\}$ and $r_{i} = d/(a_{i}+2b_{i})$. Then the mixed function germ $\fab$ is weighted homogeneous of type $(r_{1},\dots, r_{n};d,d)$ and deformations $f_{t}(z,\cj{z}) = \fab + t\Theta(z,\cj{z})$, $\Theta(z,\cj{z}) = (\Theta_{1},\Theta_{2})$ with $\fl(\Theta_{i}) > d - r_{n} + r_{1}$, for $i=1,2$, are bi-Lipschitz trivial.
    \end{Lem}
        \begin{proof}
            We shall see that the origin is an isolated singularity of $\fab$ and then apply Theorem \ref{def}. Let $z_{i} = (x_{i},y_{i})$, where $x_{i},y_{i} \in \re^{2}$, then
        \begin{align*}
            x_{i} &= \frac{1}{2}(z_{i} + \cj{z}_{i}), \\
            y_{i} &= \frac{1}{2}(z_{i} - \cj{z}_{i}).
        \end{align*}
        Therefore, for $g_{i} = z_{i}^{a_{i}}\nm z_{i}\nm^{2b_{i}} = z_{i}^{a_{i}+b_{i}}\cj{z}_{i}^{b_{i}}$, the real and imaginary parts are:
        \begin{align*}
            2\Re(g_{i}) &= \nm z_{i} \nm^{2b_{i}}(z_{i}^{a_{i}} + \cj{z}_{i}^{a_{i}}),\\
            2i\Im(g_{i}) &= \nm z_{i} \nm^{2b_{i}}(z_{i}^{a_{i}} - \cj{z}_{i}^{a_{i}}).
        \end{align*}
The terms of interest will correspond to the Jacobian matrix formed from the partial derivatives of $g_{i}$ in relation to $z_{i}$ and $\cj{z}_{i}$:
        \begin{align*}
           \frac{1}{2} \cdot \begin{pmatrix}
                (a_{i}+b_{i})z_{i}^{a_{i}+b_{i}-1}\cj{z}_{i}^{b_{i}} + b_{i}z_{i}^{b_{i}-1}\cj{z}_{i}^{a_{i}+b_{i}}&  b_{i}z_{i}^{a_{i}+b_{i}}\cj{z}_{i}^{b_{i}-1} + (a_{i}+b_{i})z_{i}^{b_{i}}\cj{z}_{i}^{a_{i}+b_{i}-1}\\
                (a_{i}+b_{i})z_{i}^{a_{i}+b_{i}-1}\cj{z}_{i}^{b_{i}} - b_{i}z_{i}^{b_{i}-1}\cj{z}_{i}^{a_{i}+b_{i}} & b_{i}z_{i}^{a_{i}+b_{i}}\cj{z}_{i}^{b_{i}-1} - (a_{i}+b_{i})z_{i}^{b_{i}}\cj{z}_{i}^{a_{i}+b_{i}-1}
            \end{pmatrix}
        \end{align*}
        whose determinant is $1/2\left[b_{i}^{2}-(a_{i}+b_{i})^{2}\right]\nm z_{i} \nm^{2(a_{i}+2b_{i}-1)}$. From the conditions on $a_{i}$ and $b_{i}$, this determinant is zero if and only if $z_{i} = 0$. 
    \end{proof}
    
    \begin{The}\label{deffam}
    Let $f(w,z) = p(w,\cj{w}) + \sum_{i=1}^{n}z_{i}^{a_{i}}q_{i}$ be a mixed function germ such that:
    \begin{enumerate}
        \item $p(w,\cj{w})$ is a real mixed function with $j^{l}(p) = \sum_{i=1}^{m}\nm w_{i} \nm^{2l_{i}}$, for integers $l > m$ with $l$ sufficient large and $l_{i} \ge 1$ for all $i$.
        \item $q_{i}$ is a real mixed polynomial of the form $q_{i}(z,\cj{z}) = \nm z_{i} \nm^{2b_{i}} + \tilde{q}_{i}(z,\cj{z})$.
        \item $d = \lcm_{i}\{a_{i} + 2b_{i}\}$ and $r_{i} = d/(a_{i} + 2b_{i})$.
        \item $\fl\left(\sum_{i=1}^{n}\tilde{q}_{i}\right) > d-r_{n}+r_{1}$.
    \end{enumerate}
 Then $f$ is bi-Lipschitz-equivalent to $h = \sum_{i=1}^{m}\nm w_{i} \nm^{2l_{i}} + \fab(z,\cj{z})$. Moreover, $h$ is topologically equivalent to $g = \sum_{i=1}^{m}\nm w_{i} \nm^{2l_{i}} + \ga(z,\cj{z})$. In particular, if $a_{1} = 1$, then $f$ is a topological submersion.
    \end{The}   
        \begin{proof}
           The assertion follows straightforwardly from Lemma \ref{l0} and Theorem \ref{def} applied to $p(w,\cj{w})$ and then to $\sum_{i=1}^{n}z_{i}^{a_{i}}q_{i}$ and also by considering the topological normal forms in Proposition \ref{topfor}.
    \end{proof}
    Thus, we obtain the topological normal form $h = \sum_{i=1}^{m}\nm w_{i} \nm^{2l_{i}} + \sum_{i=1}^{n}z_{i}^{a_{i}}$ for $f$. As a consequence of Proposition \ref{corfam}, we have the following results.
    
    \begin{Cor}
        Let $\abf$ be fixed, where $a_{i} \ge 1$ for all $i = 1, \dots, n$. Let $\Lambda_{\abf}$ be the family of mixed functions $h_{\abf,\textbf{b}} = \sum_{i=1}^{n}z_{i}^{a_{i}}q_{i}$ such that $q_{i}$ is a real mixed function as in Theorem \ref{deffam}. There exist subfamilies of $\Lambda_{\abf}$ that are topologically trivial and contain infinitely many distinct Lipschitz classes.
    \end{Cor}

    \begin{Cor}
        Let $\abf$ be fixed, where $a_{i} \ge 1$ for all $i = 1, \dots, n$. Let $\Lambda_{\abf}$ be the family of mixed functions $h_{\abf,\textbf{b}} = p(w,\cj{w}) + \sum_{i=1}^{n}z_{i}^{a_{i}}q_{i}$ such that $p$ and $q_{i}$ are real mixed functions as in Theorem \ref{deffam}. There exist subfamilies of $\Lambda_{\abf}$ that are topologically trivial and contain infinitely many distinct Lipschitz classes.
    \end{Cor}

\section{Mixed surfaces}\label{mixsuf}

In this section, we consider the problem of comparing the Lipschitz geometry of the mixed surfaces $\fab^{-1}(0)$, where $\fab: (\co^{2},0) \longrightarrow (\co,0)$. The first step is to determine the tangent cones, which are 2-dimensional in most cases. This yields conditions for inner, outer, and ambient equivalences. Throughout this section, we shall denote $\fab^{-1}(0) = \Vab$ and $\ga^{-1}(0) = \Va$. Moreover, for the mixed surface $\Vab$ we define $$\muab := \frac{a_{2}+2b_{2}}{a_{1}+2b_{1}}.$$ We shall prove that this number is an invariant of the subanalytic outer bi-Lipschitz geometry of a large class of mixed Pham-Brieskorn surfaces.

\begin{Lem}\label{tgcsup}
    Let $b_{1},b_{2} \ge 0$ and $a_{2} \ge 1$. The following statements hold:
    \begin{enumerate}
        \item If $a_{1} \ge 1$ and $\muab > 1$, then $C(\Vab,0) = \{ z_{1} = 0\}$.
        \item If $a_{1} \ge 1$ and $\muab = 1$, then $C(\Vab,0) = \Vab$.
        \item If $a_{1} = 0$ and $\muab > 1$, then $C(\Vab,0) \subset \left\{\Im z_{2}^{a_{2}} = 0\right\} \cap \{ z_{1} = 0\}$.
        \item If $a_{1} = 0$ and $\muab < 1$, then $C(\Vab,0) = \{ z_{2} = 0\}$.
        \item If $a_{1} = 0$ and $\muab = 1$, then $C(\Vab,0) = \Vab$.
    \end{enumerate}
    We say that $\Vab$ is of type (1), (2), (3), (4), or (5) according to each item above.
\end{Lem}
    \begin{proof}
        By Proposition \ref{rem0} and Remark \ref{rem1}, one inclusion is clear in all cases. Recall the characterization in Lemma \ref{chale}. For item (1), let $v = (0,v_{2}) \in \{z_{1}=0\}$ and $(t_{k}) \subset \re$ any sequence of positive integers such that $t_{k} \longmapsto 0$. Since $\Vab\setminus\{0\} \subset \co^{*2}$ the distance of $t_{k}v$ to $\Vab$ is bounded by the norm of the corresponding $z_{1}$-coordinate, which is $\nm z_{1} \nm = t_{k}^{\muab}\nm v_{2} \nm^{\muab}$, where $a_{1} \ge 1$. Therefore
        \begin{align*}
            \lim_{k\to \infty} \frac{d(t_{k}v, \Vab)}{t_{k}} \le \lim_{k \to \infty} \frac{\left(t_{k}\nm v_{2}\nm\right)^{\muab}}{t_{k}} = 0,
        \end{align*}
        since $\muab > 1$ and similarly for the item (4), but considering the reverse inequality $\muab < 1$. Let us consider items (2) and (5). Since $\Vab$ is homogeneous, if $v \in \Vab$, then for any sequence $(t_{k}) \subset \re$ it holds that $t_{k}v \in \Vab$. Thus, the limit above is zero, and we conclude.  
    \end{proof}

\begin{Rem}
    \normalfont It holds that $\Vab$ is a regular manifold at the origin if and only if $a_{1} = 1$ and $b_{1} = 0$ or $a_{2} = 1$ and $b_{2} = 0$. Henceforth, we exclude these cases from the statements. For surfaces $\Vab$ of type (1), we also suppose that $\gcd(a_{1}, a_{2}) = 1$, and then $\Vab$ is homeomorphic with an irreducible complex Pham-Brieskorn variety $\Va$. This implies that $\Vab \setminus \{0\}$ is connected. 

\end{Rem}

\begin{Prop}\label{l17} Let $\Vab$ be a mixed Pham-Brieskorn surface.
    \begin{enumerate}
        \item If $\Vab$ is of type (2) or (5), then it is Lipschitz normally embedded.
        \item If $\Vab$ is of type (3), then $\Vab$ is Lipschitz normally embedded.
        \item If $\Vab$ is of type (4) with $a_{2} =1$, then $\Vab$ is Lipschitz normally embedded.
        \item If $\Vab$ is of type (4) with $a_{2} \ge 2$, then $\Vab$ is not Lipschitz normally embedded.
        \item If $\Vab$ is of type (1) and $a_{1}+2b_{1}$ is an even number, then it is not Lipschitz normally embedded.
    \end{enumerate}
\end{Prop}
    \begin{proof}
        Mixed surfaces of types (2) or (5) are of homogeneous type with an isolated singularity at the origin, so the statement follows from \cite[Corollary 2.10]{kerner}. We apply the arc criterion \cite[Theorem 2.2]{birmen} for the other statements. First, let us consider types (3) and (4). For $z_{i} \in \co$, we take real coordinates $z_{i} = (x_{i}, y_{i}) \in \re^{2}$, where $i=1,2$. Up to the coordinate linear change given by $w_{2} = \xi z_{2}$, where $\xi^{a_{2}} = -1$, we suppose $\fab = \nm z_{1} \nm^{2b_{1}} - z_{2}^{a_{2}}\nm z_{2} \nm^{2b_{2}}$. The surfaces $\Vab$ are determined as an intersection with $\Im z_{2}^{a_{2}} = 0$, which consists of a finite union of hyperplanes in $\co^{2}$ passing through the origin. For each one of these lines, we take coordinates $(x_{2},y_{2})$ such that $\lambda x_{2} = y_{2}$, where $\lambda \in \re$, and substituting it in the equation $\nm z_{1}\nm^{2b_{1}} - \Re z_{2}^{a_{2}}\nm z_{2} \nm^{2b_{2}} = 0$, we conclude that $\Vab$ is a union of $\beta$-horns intersecting only at the origin, where $\beta = (a_{2}+2b_{2})/2b_{1}$. Notice that $\beta > 1$ and $\beta < 1$ correspond to types (3) and (4), respectively. In particular, in case (4), the components are smooth manifolds. Thus, if $a_{2} = 1$, $\Vab$ is a smooth manifold or a $\muab$-horn and it is normally embedded.

        Let us consider the case $a_{2} \ge 2$ for surfaces of type (4), in which the tangent cone of each component is 2-dimensional. Let $\lambda_{i}x_{2} = y_{2}$ for $i=1,2$, where $\lambda_{1} \neq \lambda_{2}$. Consider the following curves in distinct components of $\Vab\setminus \{0\}$:
        \begin{align*}
            \gamma_{1} : y_{1}^{2b_{1}} - x_{2}^{a_{2}+2b_{2}}\xi_{1} &= 0, \\
            \gamma_{2} : y_{1}^{2b_{1}} - x_{2}^{a_{2}+2b_{2}}\xi_{2} &= 0,
        \end{align*}
        where $\xi_{i} = \Re\left(1 + \ibf \lambda_{i}\right)^{a_{2}}\nm 1 + \ibf \lambda_{i} \nm^{2b_{2}}$. We may parametrize these curves by:
        \begin{align*}
            \gamma_{1}(t) &= (0, t, \tilde{\xi}_{1}t^{\frac{1}{\muab}},\lambda_{1}\tilde{\xi}_{1}t^{\frac{1}{\muab}}), \\
            \gamma_{2}(t) &= (0, t, \tilde{\xi}_{2}t^{\frac{1}{\muab}},\lambda_{2}\tilde{\xi}_{2}t^{\frac{1}{\muab}}),
        \end{align*}
        where $\tilde{\xi}_{i} = \xi_{i}^{-1/a_{2}+2b_{2}}$. Notice that if $a_{2}$ is even, then $\xi_{i} > 0$ and the expressions above are well-defined. Any path $\alpha_{t}(s)$ connecting the points $\gamma_{1}(t)$ and $\gamma_{2}(t)$ in these curves passes through the origin. Moreover, one has that the order of $\nm \gamma_{1}(t) - \gamma_{2}(t) \nm$ is $1/\muab$ and the order of $\nm \gamma_{1}(t) \nm + \nm \gamma_{2}(t)\nm$ is $1$. Since $\muab < 1$, the following limit diverges, which implies that $\Vab$ is not Lipschitz normally embedded:
        \begin{align*}
            \lim_{t \to 0} \frac{d_{i}(\gamma_{1}(t), \gamma_{2}(t))}{\nm \gamma_{1}(t) - \gamma_{2}(t)\nm} \ge \lim_{t \to 0} \frac{\ell(\alpha_{t})}{\nm \gamma_{1}(t) - \gamma_{2}(t)\nm } \ge \lim_{t \to 0} \frac{\nm \gamma_{1}(t) \nm + \nm \gamma_{2}(t) \nm}{\nm \gamma_{1}(t) - \gamma_{2}(t)\nm},
        \end{align*}
        where $\ell(\alpha)$ denotes the length of $\alpha$.
        
        For surfaces of type (3), all components have one-dimensional tangent cones that intersect only at the origin. We need to restrict to the case of distinct components $S_{1}$ and $S_{2}$. We claim that there exists $C > 0$ such that $\nm x \nm + \nm y \nm \le C\nm x-y \nm$ for all $x \in S_{1}$ and $y \in S_{2}$. On the contrary, there exist sequences $(p_{k}) \in S_{1}$ and $(q_{k}) \in S_{2}$ such that $\lim_{k}p_{k} = \lim_{k}q_{k} = 0$ and
        \begin{align*}
            \nm p_{k} \nm + \nm q_{k} \nm > c_{k}\nm p_{k} - q_{k} \nm,
        \end{align*}
        where $c_{k} > 0$ for all $k \in \mathbb{N}$ and $\lim_{k}c_{k} = \infty$. Let $t_{k} = \nm p_{k} \nm + \nm q_{k} \nm$ and define $x_{k} = \frac{1}{t_{k}}p_{k}$ and $y_{k} = \frac{1}{t_{k}}q_{k}$. Since these sequences are bounded, up to pass to a subsequence, there exist $v_{1}$ and $v_{2}$ in the tangent cones of $S_{1}$ and $S_{2}$, respectively, such that $\lim_{k}x_{k} = v_{1}$ and $\lim_{k}y_{k} = v_{2}$. Moreover, $\nm v_{1} \nm + \nm v_{2} \nm = 1$ and $v_{1}$ and $v_{2}$ are not both vanishing. On the other hand, 
        \begin{align*}
            v_{1} - v_{2} = \lim_{k} \frac{x_{k}-y_{k}}{t_{k}} < \lim_{k} \frac{1}{c_{k}} = 0,
        \end{align*}
        which contradicts the fact that the tangent cones of $S_{1}$ and $S_{2}$ intersect only at the origin. Let $\gamma_{i} \subset S_{i}$ be analytic curves in distinct components, for $i=1,2$. In polar coordinates, we may parametrize these paths as follows:
        \begin{align*}
            \gamma_{i}(t) = \left( r_{i}(t)^{\muab}e^{\ibf \theta_{i}(t)}, r_{i}(t)e^{\frac{\ibf 2\pi l_{i}}{a_{2}}}\right),
        \end{align*}
        for $i=1,2$ and $l_{1} \neq l_{2}$. For a fixed $t$, a path $\alpha_{i,t}$ connecting $\gamma_{i}(t)$ to the origin is given by:
        \begin{align*}
            \alpha_{i,t}(s) = \left( s^{\muab}e^{\ibf \theta_{i}(t)}, se^{\frac{\ibf 2\pi l_{i}}{a_{2}}}\right),
        \end{align*}
        where $s \in [0,r_{i}(t)]$ for $i=1,2$. It follows that $\ell(\alpha_{i,t}) = r_{i}(t) + o(r_{i}(t))$, where $o(r_{i}(t))$ denotes terms of higher order in $r_{i}(t)$ and $\ell$ denotes the length of the curve. Hence $\ell(\alpha_{1,t}) + \ell(\alpha_{2,t})$ and $\nm \gamma_{1} \nm + \nm\gamma_{2}\nm$ have the same order in $t$. Thus, the following limit is bounded:
        \begin{align*}
            \lim_{t \to 0} \frac{d_{i}(\gamma_{1}(t), \gamma_{2}(t))}{\nm \gamma_{1}(t) - \gamma_{2}(t)\nm} \le \lim_{t \to 0} \frac{\ell(\alpha_{1,t}) + \ell(\alpha_{2,t})}{\nm \gamma_{1}(t) \nm + \nm\gamma_{2}(t)\nm}\cdot \frac{\nm \gamma_{1}(t) \nm + \nm\gamma_{2}(t)\nm}{\nm \gamma_{1}(t)-\gamma_{2}(t)\nm} \le C',
        \end{align*}
        for some $C' > 0$. This implies that the inner and outer distances have the same order, and the result follows from the arc criterion.
        
        For mixed surfaces of type (1), up to the coordinate linear change given by $w_{2} = \xi z_{2}$, where $\xi^{a_{2}} = -1$, we suppose $\fab = z_{1}^{a_{1}}\nm z_{1} \nm^{2b_{1}} - z_{2}^{a_{2}}\nm z_{2} \nm^{2b_{2}}$. Then, $\Vab$ contains the curve $\gamma$ defined by $x_{1}^{a_{1}+2b_{1}} - x_{2}^{a_{2}+2b_{2}} = 0$ whose branches can be parameterized by $\gamma^{\pm} = (\pm t^{\muab},t)$, where $t \in \re$ is positive and $\muab = (a_{2}+2b_{2})/(a_{1}+2b_{1}) > 1$. The outer distance $\nm \gamma^{+}(t) - \gamma^{-}(t) \nm$ has order $\muab$. Fix $t_{0} > 0$ and let $\lambda(s)$ be a path connecting the points $\gamma^{+}(t_{0})$ and $\gamma^{-}(t_{0})$, respectively. If $\lambda(s)$ contains the origin, it is clear that the length of $\lambda$ satisfies $\ell(\lambda)$$ \ge 2t_{0}$. On the contrary, we apply an argument used in \cite[Proposition 3.4]{samle}. The map $\pi : \Vab\setminus \{0\} \longrightarrow \co^{*}$ given by $\pi(z_{1},z_{2}) = z_{2}$ is a Lipschitz $a_{1}$-finite covering map. It follows that $\pi(\lambda(s)) \subset \co^{*}$ is a loop based on $t_{0}$, which will denote by $\delta$ (see Figure \ref{fig2}). Furthermore, since $\pi^{-1}(t_{0})$ contains a finite number of points, $\delta$ is not homotopically trivial. This shows that the inner distance $\ell(\lambda(s)) \ge \ell(\delta) \ge 2t_{0}$, where $\ell(\lambda(s))$ and $\ell(\delta)$ denote the length of $\lambda$ and $\delta$, respectively. In both cases, we conclude that, for a sufficiently small $t> 0$, $d_{i}(\gamma^{+}(t),\gamma^{-}(t)) \ge 2t$. This implies that the limit       
        \begin{align*}
            \lim_{t\to 0}\frac{d_{i}(\gamma^{+}(t), \gamma^{-}(t))}{\nm \gamma^{+}(t) - \gamma^{-}(t) \nm}
        \end{align*}
        diverges and $\Vab$ is not normally embedded.

        \begin{figure}[ht!]
        \centering
        \includegraphics[scale=.75]{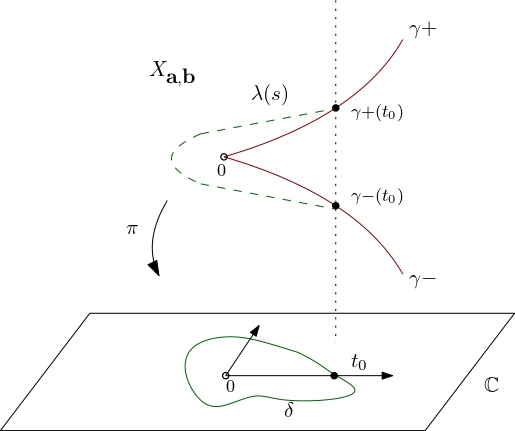}
        \caption{Projection of $\lambda(s)$ on $\co^{*}$}
        \label{fig2}
        \end{figure}
    \end{proof}
    
    See also \cite[Corollary 4.3]{tibar} for the non-normal embedding property of complex Pham-Brieskorn varieties in any dimension. Taking into account the description in the proof of Proposition \ref{l17}, one has the following.

    \begin{Prop}\label{p1}
        Each component of a mixed Pham-Brieskorn surface of type (1), (2), (4), or (5) is inner bi-Lipschitz equivalent to the metric cone $H_{1}$.
    \end{Prop}
    \begin{proof}
        Recall Theorem \ref{inde1} and notice that each component of the mixed surface germ has an isolated singularity at the origin under the hypothesis above. By the previous lemma, $\Vab$ has a 2-dimensional tangent cone, thus we can choose semialgebraic arcs with linearly independent tangent vectors. This implies that the order of contact is 1, and we conclude.
    \end{proof}

\begin{Exam}
    \normalfont
    Proposition \ref{p1} shows that mixed surfaces of type (3) are the only interesting cases from the inner geometry point of view. Indeed, let $\beta \ge 1$ be a rational number and $b,d$ non-negative integers such that $\beta = (2d+1)/2b$. Then the $\beta$-horn can be described as $h_{b,d}^{-1}(0)$, where $h_{b,d} = \nm z_{1} \nm^{2b} + z_{2}^{1+d}\cj{z}_{2}^{d}$. Given a fixed $\beta$, any choice of $b,d$ as before provides inner bi-Lipschitz equivalent sets, because this is a complete invariant.

    Moreover, the mixed function $h_{b,d}$ is topologically equivalent to the submersion $\nm z_{1} \nm^{2b} + z_{2}$. The results of Section \ref{lipeq} show that this does not extend to bi-Lipschitz equivalence. Alternatively, note that the tangent cone of a horn is a line segment while the tangent cone of a submersion is a plane (see Figure \ref{fig}). Thus, Theorem \ref{tsam} yields a second reason for this latter claim.    

    \begin{figure}[ht!]
    \centering
    \includegraphics[scale=0.7]{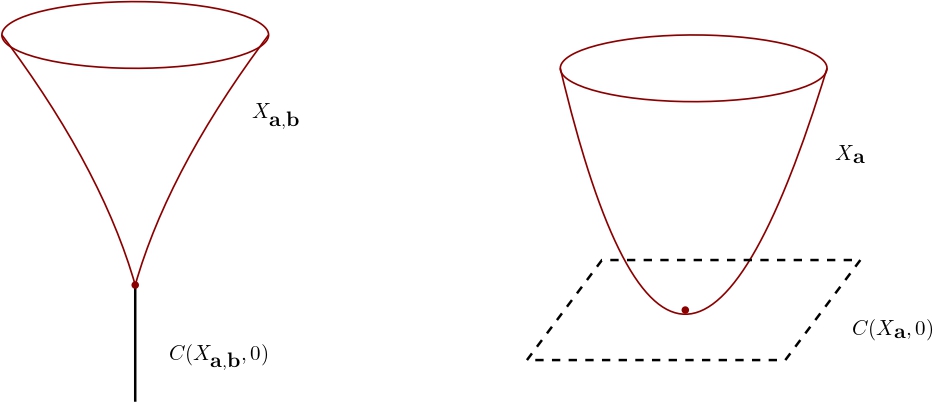}
    \caption{Mixed surfaces and the respective tangent cones for $a_{1} = 0$}
    \label{fig}
    \end{figure}
\end{Exam}

\begin{Prop}\label{p2}
     Under the hypothesis on surfaces of type (1) in Proposition \ref{l17}, the mixed Pham-Brieskorn surfaces of types (1) and (3) are preserved under outer bi-Lipschitz equivalence.
\end{Prop}
    \begin{proof}
        Let $\Vab, \Vcd$ be mixed Pham-Brieskorn surfaces, where $\Vab$ is of type (1). If $\Vab$ and $\Vcd$ are outer bi-Lipschitz equivalent, then the non-normal embedding property is preserved. This implies that $\Vcd$ is of type (1), (3), or of type (4) with $c_{2} \ge 2$. Since the tangent cones of mixed surfaces of types (1) and (3) have distinct dimensions and $\Vab\setminus\{0\}$ would have more than one component in the second case, $\Vcd$ is also of type (1). The same reasoning applies if $\Vab$ is of type (3) and we conclude.
    \end{proof}
In general, there is no geometric relation with respect to the outer metric between mixed surfaces and their topological normal forms, as shown in the next corollary.
\begin{Cor} \label{su3}
    Under the hypothesis on surfaces of type (1) in Proposition \ref{l17}, let $\Vab$ be a mixed Pham-Brieskorn surface with $b_{1}, b_{2} \neq 0$. Suppose that:
    \begin{enumerate}
        \item $a_{1} = a_{2}$ and $\muab > 1$; or
        \item $a_{1} < a_{2}$ and $\muab = 1$.
    \end{enumerate}
   Then $\Vab$ is not outer bi-Lipschitz equivalent to its topological normal form $\Va$.
\end{Cor}

A more general criterion can be stated as follows.

\begin{The}\label{radial}
    Let $\Vab$ and $\Vcd$ be two mixed Pham-Brieskorn surfaces of type (1) such that $a_{1} \ge 2$ and $\mucd < \muab$. Then there does not exist a subanalytic bi-Lipschitz homeomorphism $\phi : (\Vab,0) \longrightarrow (\Vcd,0)$.
\end{The}
    \begin{proof}
        The proof is based on the argument of \cite[Example 2.1]{Fernandes2003}. Let $r \ge 0$ and define the following branches in $\Vab$:
        \begin{align*}
           \gamma_{1}(r) = \left( r^{\muab}, re^{i\frac{\pi}{a_{2}}}\right) \quad \text{and} \quad \gamma_{2}(r) = \left( r^{\muab}e^{i\frac{2\pi}{a_{1}}}, re^{i\frac{\pi}{a_{2}}}\right).
        \end{align*}
        Next, we look at the inner distance of these arcs. Let $\lambda(t)$ be a minimizing path connecting $\gamma_{1}(r)$ and $\gamma_{2}(r)$, where $r$ is fixed. We may apply the same argument in the proof of the last item of Proposition \ref{l17} to conclude that the projection of $\lambda(t)$ on the $z_{2}$-axis is a non-trivial loop based on $re^{i\frac{\pi}{a_{2}}}$ whose length is bounded from below by $2r$. Therefore, the limit
        \begin{align}\label{inndt}
            \lim_{r \to 0} \frac{d_{i}(\gamma_{1}(r),\gamma_{2}(r))}{r} \ge 2.
        \end{align}
    
        Suppose the existence of such a bi-Lipschitz map $\phi : (\Vab,0) \longrightarrow (\Vcd,0)$. It follows that we may parametrize the images of $\gamma_{1}, \gamma_{2}$ by
        \begin{align*}
            \sigma_{k}(r) = \left(r^{\mucd}e^{i\frac{\omega_{k}(r)c_{2}-\pi}{c_{1}}}, re^{i\omega_{k}(r)}\right) \in \phi\left(\gamma_{k}\right) \cap \mathbb{S}^{3}_{r},
        \end{align*}
        where $k=1,2$. The tangency order is preserved by \cite[Theorem 1.1]{Fernandes2003}, therefore
        \begin{align}\label{eqex}
            \text{tord}_{r}\nm \sigma_{1}(r) - \sigma_{2}(r)\nm = \muab,
        \end{align}
        where $\text{tord}$ denotes the tangency order. On the other hand, we claim that $\lim_{r \to 0} \omega_{1}(r) - \omega_{2}(r) = 2\pi l$ for some integer $l$ such that $lc_{2} \equiv 0 \mod c_{1}$. Otherwise, we get
        \begin{align*}
            \text{tord}_{r}\nm \sigma_{1}(r) - \sigma_{2}(r)\nm = 1,\quad &\text{if} \; \lim_{r \to 0} \omega_{1}(r) - \omega_{2}(r) \notin 2\pi\mathbb{Z}, \\
            \text{tord}_{r}\nm \sigma_{1}(r) - \sigma_{2}(r)\nm \le \mucd, \quad &\text{if} \; \lim_{r \to 0} \omega_{1}(r) - \omega_{2}(r) = 2\pi l \;\; \text{for} \;\; lc_{2} \not\equiv 0 \mod c_{1}, 
        \end{align*}
        which is not possible by \eqref{eqex} and the fact that $\muab > \mucd$.  In this case, up to a reparametrization of the curves $\sigma_{k}$, we may suppose $l = 0$.
        
        Let us consider again the inner contact of $\sigma_{1}$ and $\sigma_{2}$. A path $\rho(t) \subset \Vcd$ connecting these curves is
        \begin{align*}
            \rho(t) = \left(r^{\mucd}e^{i\frac{\kappa(t)c_{2}-\pi}{c_{1}}}, re^{i\kappa(t)}\right),
        \end{align*}
        where $\kappa(t) = t\omega_{1}(r) + (1-t)\omega_{2}(r)$ and $t \in [0,1]$. The condition $\lim_{r \to 0}\omega_{1}(r) - \omega_{2}(r) = 0$ now implies that the length $\ell(\rho(t))$ has order higher than $1$. Therefore,
        \begin{align*}
            \lim_{r \to 0} \frac{d_{i}(\gamma_{1}(r),\gamma_{2}(r))}{r} \le \lim_{r \to 0} \frac{\ell(\rho(t))}{r} = 0,
        \end{align*}
        which contradicts \eqref{inndt}, provided that $\phi$ is bi-Lipschitz and also preserves the inner distance.    
    \end{proof}
Let $\Vab$ and $\Vcd$ be two mixed surfaces of type (3). We already know that these are unions of $\muab$ and $\mucd$-horns, respectively. Since this index is a complete invariant, these are outer bi-Lipschitz equivalent if and only if $\muab = \mucd$. Concerning the ambient equivalence, we have the following.
\begin{Prop}
     Under the hypothesis on surfaces of type (1) in Proposition \ref{l17}, let $\Vab$ and $\Vcd$ be two ambient bi-Lispchitz equivalent mixed Pham-Brieskorn surfaces such that $a_{1}, c_{1} \ge 1$. If $\Vab$ is of type (1) or (2), then $\Vab$ and $\Vcd$ are of the same type.
\end{Prop}
    \begin{proof}
        The topological types of $\Vab$ and $\Vcd$ are the same, and Theorem \ref{etsu} applies. Moreover, the ambient equivalence implies the outer, and we apply Proposition \ref{p2}.
   \end{proof}
The ambient geometry of $\Vab$ is related to $\Va$ as follows.
\begin{Prop}\label{su5}
    Let $\Vab$ be a mixed Pham-Brieskorn surface of type (1) or (2) which is ambient bi-Lipschitz equivalent to a complex analytic curve germ in $\co^{2}$. Then $\Vab$ and its topological normal form $\Va$ are ambient bi-Lipschitz equivalent.
\end{Prop}
    \begin{proof}
    Each complex analytic curve as in the statement has the same topological type as $\ga$ and, by \cite[Theorem 2.6]{pich}, this is a complete invariant.
    \end{proof}
In particular, a mixed surface as in Corollary \ref{su3} is never ambient bi-Lipschitz equivalent to a complex analytic plane curve germ.

\section*{Acknowledgments}

The author is grateful to his thesis advisor, Professor Maria Ruas (in memoriam), for helpful discussions and comments. He
also thanks Davi Medeiros for insightful conversations and for suggesting the argument applied in Theorem 6.8. This
work was partially completed during a sandwich doctorate program at the Instituto de Matemáticas of the Universidad
Nacional Autónoma de México at Cuernavaca, and the author would like to thank its community for their hospitality and
support.


\subsection*{Conflict of interest} The author declares that there is no conflict of interest.

\subsection*{Data Availability} Data availability is not applicable to this article.

        \bibliographystyle{siam}
	\bibliography{biblio} 
	\addcontentsline{toc}{section}{References}

\end{document}